\documentclass[11pt,a4paper]{amsart}
\usepackage{amsmath}
\usepackage{amsfonts}
\usepackage{graphicx}
\usepackage{framed}
\usepackage{amssymb}
\usepackage{esvect}

\usepackage{etex}
\usepackage{latexsym}
\usepackage[english]{babel}
\usepackage[utf8x]{inputenc}
\usepackage{xcolor}

\def \N {\mathbb{N}}
\def \R {\mathbb{R}}

\def \d {\mathrm{d}}

\def \longto {\longrightarrow}
\def \d {\mathrm{d}}

\def \FF {\mathcal{F}}
\def \Tv {\mathrm{T}}

\def \de {\partial}

\def \HH {\mathcal{H}}

\usepackage{amsthm}
\theoremstyle{definition}
\newtheorem{definition}{Definition}[section]
\newtheorem*{axiom}{Assumptions}
\newtheorem{example}[definition]{Example}
\newtheorem{remark}[definition]{Remark}

\theoremstyle{plain}

\newtheorem{theorem}[definition]{Theorem}
\newtheorem{proposition}[definition]{Proposition}

\numberwithin{equation}{section}
\begin{document}
 \title[Pohozaev-type identities]{Pohozaev-type identities for differential operators driven by homogeneous vector fields}
 
\author[S.\ Biagi]{Stefano Biagi}
\author[A.\ Pinamonti]{Andrea Pinamonti}
\author[E.\ Vecchi]{Eugenio Vecchi}

\address[S.\,Biagi]{Politecnico di Milano
 \newline\indent Dipartimento di Matematica \newline\indent
 Via Bonardi 9, 20133 Milano, Italy}
 \email{stefano.biagi@polimi.it}
 
 \address[A.\ Pinamonti]{Dipartimento di Matematica \newline\indent
	Universit\`a degli Studi di Trento,
	Via Sommarive 14, 38123, Povo (Trento), Italy}
\email{andrea.pinamonti@unitn.it}
 
 \address[E.\,Vecchi]{Politecnico di Milano
 \newline\indent Dipartimento di Matematica \newline\indent
 Via Bonardi 9, 20133 Milano, Italy}
 \email{eugenio.vecchi@polimi.it}


\subjclass[2010]
{}

\keywords{sub-elliptic semilinear equations; non-existence results; Pohozaev-type identities; geometric methods for boundary-value problems.} 
 
 \thanks{The authors are members of INdAM. S. Biagi
is partially supported by the INdAM-GNAMPA 2020 project 
\emph{Metodi topologici per problemi al contorno associati a certe 
classi di equazioni alle derivate parziali}.
A. Pinamonti and E. Vecchi are partially supported
by the INdAM-GNAMPA 2020 project 
\emph{Convergenze variazionali per funzionali e operatori dipendenti da campi vettoriali}}.

 \date{\today}
 \begin{abstract}
 We prove Pohozaev-type identities for smooth solutions of Euler-La\-gran\-ge equations of 
 second and fourth order that arise from functional de\-pen\-ding on homogeneous 
 H\"{o}rmander vector fields. We then exploit such integral 
 identities to prove non-existence results for the associated boundary value problems.
 \end{abstract}
  \maketitle
 \section{Introduction}
 In 1965 in \cite{Poho} Pohozaev proved an integral identity 
 for solutions of the following elliptic boundary value problem:
 \begin{equation}\label{eq:POHO}
 \left\{\begin{array}{rl}
  -\Delta u = f(u) & \textrm{in } \Omega,\\
  u=0 & \textrm{on } \partial \Omega,
  \end{array}\right.
 \end{equation}
 \noindent where $\Omega \subset \mathbb{R}^n$ is an open and bounded 
 set with smooth boundary $\partial\Omega$.
 In particular, under the geometric assumption on the set 
 $\Omega$ of being star-shaped, he was able to show that \eqref{eq:POHO} does not admit non-trivial  
 solutions in $C^2(\Omega) \cap C^1(\overline{\Omega})$.
 In 1986 in \cite{PucSer}. Pucci and Serrin extended the approach of Pohozaev proving integral identities 
 for solutions of a large class of variational PDEs 
 coming from functionals possibly depending on the Hessian. 
 
 The starting idea of Pohozaev goes actually 
 back to Rellich and Nehari \cite{Rellich,Nehari} and can be 
 summarized as follows: given a sufficiently smooth 
 solution of \eqref{eq:POHO}, it suffices to multiply  the equation 
 $-\Delta u = f(u)$ by $x \cdot \nabla u$, integrate over $\Omega$ 
 and apply the Divergence Theorem. This will lead to the celebrated Pohozaev identity
 \begin{equation}\label{eq:POHO2}
  \dfrac{n-2}{2} \int_{\Omega}\|\nabla u\|^2 \,\d x - n \int_{\Omega}F(u)\,\d x 
  + \dfrac{1}{2}\int_{\partial   
  \Omega}\left| \dfrac{\partial u}{\partial\nu}\right|^2 \, 
  \langle x ,\nu\rangle\,\d x = 0,
\end{equation}
 \noindent where $F$ is a primitive of $f$ and $\nu$ denotes the unit outward normal of $\partial\Omega$.
 From \eqref{eq:POHO2} it is pretty easy to get nonexistence results (of non-trivial solutions) under 
 appropriate assumptions on the nonlinearity $f$, and therefore on $F$ itself. 
 Assume, e.g., that $F(u)\leq 0$. Then
 \begin{equation*}
 0 \geq n\int_{\Omega} F(u)\,\d x = \dfrac{n-2}{2} \int_{\Omega}\|\nabla u\|^2\,\d x 
 + \dfrac{1}{2} 
 \int_{\partial \Omega}\left| \dfrac{\partial u}{\partial\nu}\right|^2 \,
 \langle x, \nu \rangle\,\d x.
 \end{equation*}
 \noindent 
 Since
 the first integral on the right hand side 
 (r.h.s., in  
 short) is non-negative, it is now clear to the entire sign of 
 the r.h.s.\,depends on $x \cdot \nu$; the latter is 
 a purely geometric quantity, only depending on the set $\Omega$ and not 
 on the solution $u$. Therefore, 
 assuming for example that $\Omega$ is star-shaped, one can achieve 
 the famous non-existence result of Pohozaev.\medskip

 Since the paper of Pucci and Serrin \cite{PucSer}, the intimate 
 connection between integral identities of   
 Pohozaev-type and non-existence results has been the object of 
 study of many papers, who has extended the 
 ideas previously recalled to cover an always wider class of PDEs. 
 The streamline has been definitely interested in extending Pohozaev’s 
 results to more general equations, such as 
 quasi-linear elliptic equations, polyharmonic equations and fractional 
 differential equations.   Without any attempt of completeness, we refer to 
 \cite{Guedda, Dinca} for the case of the $p$-laplacian, to \cite{Mit} for 
 higher order differential operators and to \cite{FallWeth, RosSerra} for 
 more recent contributions dealing with nonlocal operators.
 We must however remind 
 that there has been a certain interest also in studying non-existence results in  
 case of
 more general domains, see e.g. \cite{Dancer, McGough, Cowan} and the 
 references therein.  
 We also refer to \cite{Wagner, XuLuWang} for slightly different approaches to the proof of the classical 
 Pohozaev identity previously recalled.
 
 Another interesting line of research moved to consider non-Euclidean ambient spaces, 
 like Riemannian manifolds (see e.g. \cite{BozMit,BozMit2}) and Carnot groups, which are the prototypical 
 examples of sub-Riemannian manifolds. In this setting Pohozaev-type 
 identities, and the related non-existence reusults for certain classes 
 of semilinear subelliptic PDEs, have been established 
 in \cite{GarLanc, GarShen, GarVas, Niu}. Among several technical issues to 
 be faced in this setting, we want to stress that one has to understand how 
 to replace the star-shape assumption that naturally appears in the Euclidean case. 
 In this perspective, in \cite{GarLanc}  the authors introduced the 
 notion of $\delta_\lambda$-star-shaped set, which is closely related 
 to the anisotropic dilations defined on Carnot groups. 
 Roughly speaking, this choice allows to {\it give a sign} to the 
 quantity that naturally replaces the term $x\cdot \nu$. We address the interested reader to 
 \cite{DF,FP,DGS} for further comments on the 
 notion of star-shaped and some applications.\medskip

 The aim of this note is to continue along the line tracked in \cite{KogLanc}, where the authors proved 
 similar results for a class of differential 
 operators called $\Delta_{\lambda}$ laplacians.
 In particular, we will focus on sufficiently 
 regular solutions of second order variational PDEs that are Euler-Lagrange equations associated to 
 functionals depending on homogeneous H\"{o}rmander vector 
 fields (see Section \ref{sec.prelim} for the details). 
 The setting of the homogeneous H\"ormander vector fields
 has been studied in the series of papers 
 \cite{BGibbons, BBHeat, BBFundSol, BiBoBra, BiBoBra2, BiBra, BiLa}, 
 where several 
 \emph{global results}
 are established: in fact, this setting seems to 
 allow the possibility of developing an interesting global theory, without
 assuming the existence of a group of translations.
  A similar context has 
 been also exploited in \cite{P} to study multiplicity results for 
 solutions of possibly degenerate equation in divergence form and in \cite{MPSC1,MPSC2,MV} to study the converge of minimizers for integral functionals. \medskip
 
Following the spirit of \cite{PucSer}
we will also deduce a Pohozaev-type identity 
for smooth solutions of variational higher-order PDEs. 
To be more precise, we will start by considering functionals depending on
 the intrinsic $X$-Hessian as well, and this choice naturally lead to 
 fourth-order PDEs. As already mentioned, such kind of integral identities should
  be the good tool to study non-existence of non-trivial solutions for
  boundary value problems. Quite surprisingly, to conclude this
   kind of non-existence results it seems necessary to require 
   that the {\it full} Euclidean gradient
    vanishes on $\partial \Omega$, and not only the more intrinsic $X$-gradient 
    (see Section \ref{sec.secondorder} for a more detailed comment
    on this aspect). We believe
     that this feature is strictly related to the geometric assumption of 
     $\delta_{\lambda}$-star-shaped set.
     As a part of a future research
     project, we shall investigate more general geometric assumptions
     on $\Omega$, possibly allowing us to deal with boundary conditions
     only involving the intrinsic $X$-gradient.\bigskip

 \noindent \emph{Plan of the paper.} A short plan of the paper is now in order.
 \begin{itemize}
 \item In Section \ref{sec.prelim} 
 we recall and discuss all the basic notions needed in the proof of our main theorem
 
 \item Section \ref{sec.Orderone} is devoted to the proof of the Pohozaev identity 
 for the second order case and its application to some non-existence results.
 
 \item Finally, in Section \ref{sec.secondorder} we will comment on the case in which 
  our operator depends also on the $X$-Hessian matrix.
 \end{itemize}

 \section{Main assumptions} \label{sec.prelim}
  Throughout the sequel, we denote by
  $\mathcal{X}(\R^n)$ the Lie algebra of the smooth vector fields
  in $\R^n$. Moreover,
  if $A\subseteq\mathcal{X}(\R^n)$, we let
  $\mathrm{Lie}(A)$
  be the smallest Lie sub-algebra of $\mathcal{X}(\R^n)$ containing $A$.
  Finally, if $Y\in\mathcal{X}(\R^n)$ is of the form
  $$Y = \sum_{i = 1}^na_j(x)\frac{\de }{\de x_i}\qquad
  (\text{for some $a_1,\ldots,a_n\in C^\infty(\R^n)$}),$$
  and if $x\in\R^n$ is arbitrary, we define
  $$Y(x) := \begin{pmatrix}
   a_1(x) \\
   \vdots \\
   a_n(x)
  \end{pmatrix} \in \R^n.
  $$
  \begin{axiom}
   Let $X := \{X_1,\ldots,X_m\}\subseteq\mathcal{X}(\R^n)$ be a family of
   \emph{linearly in\-de\-pen\-dent} smooth vector fields in $\R^n$
   satisfying the following assumptions.
   \begin{description}
    \item[(H.1)] $X_1,\ldots,X_m$ are homogeneous of degree $1$
    with respect to a family of non-isotropic dilations $\{\delta_\lambda\}_{\lambda > 0}$
    in $\R^n$ of the form
    \begin{equation} \label{eq.defidela}
     \delta_\lambda(x)
    := (\lambda^{\sigma_1}x_1,\ldots,\lambda^{\sigma_n}x_n),\quad
    \begin{array}{c}
    \text{where $\sigma_1,\ldots,\sigma_n\in\N$ and} \\
    1 = \sigma_1\leq\ldots\leq\sigma_n.
    \end{array}
    \end{equation}
    We define  the $\delta_\lambda$-homogeneous
    dimension of $\R^n$ as 
    \begin{equation} \label{eq.defq}
     q := \sum_{k = 1}^n\sigma_k\geq n
    \end{equation}
    
    \item[(H.2)] $X_1,\ldots,X_m$ satisfy H\"ormander's condition at $x = 0$, that is,
    \begin{equation} \label{eq.HormanderX}
      \mathrm{dim}\Big\{Y(0):\,Y\in\mathrm{Lie}(X)\Big\} = n.
    \end{equation}
   \end{description}
  \end{axiom}
  As regards assumption (H.1), we remind that a vector field
  $Y\in\mathcal{X}(\R^n)$ is ho\-mo\-ge\-neous of degree $\alpha \in\R$
  with respect to $\{\delta_\lambda\}_{\lambda > 0}$ if
  \begin{equation} \label{eq.defiHomog}
   Y(u\circ \delta_\lambda) = \lambda^\alpha (Yu)\circ\delta_\lambda
  \quad \text{for all $u\in C^\infty(\R^n)$ and $\lambda > 0$}.
  \end{equation}
  Writing $Y = \sum_{i = 1}^n a_i(x)\de_{x_i}$ (for suitable functions
  $a_1,\ldots,a_n\in C^\infty(\R^n)$), it is easy to check that
  \eqref{eq.defiHomog} is equivalent to one of the following conditions:
  \begin{itemize}
    \item[(a)] $a_i$ is $\delta_\lambda$-homogeneous of degree $\sigma_i-\alpha$, that is,
    \begin{equation} \label{eq.delaYcoeff}
    a_i\big(\delta_\lambda(x)\big) = \lambda^{\sigma_i-\alpha}a_i(x)
    \quad \text{for all $x\in\R^n$ and $\lambda > 0$};
    \end{equation}
    \item[(b)] for every $x\in\R^n$ and every
    $\lambda > 0$ one has the identity 
    \begin{equation} \label{eq.delaYvector}
     \delta_\lambda\big(Y(x)\big) = \lambda^{\alpha}Y\big(\delta_\lambda(x)\big).
     \end{equation}
  \end{itemize}
  \begin{remark} \label{rem.conseqH}
   We list, for future reference, some easy consequences
   of as\-sum\-ptions (H.1)-(H.2) which shall be useful in the sequel.
   \begin{enumerate}
    \item It is not difficult to check that
     the combination of (H.1) and (H.2) implies the validity
     of H\"ormander's condition at every point $x\in\R^n$, i.e.,
     \begin{equation} \label{eq.HormanderXfull}
      \mathrm{dim}\Big\{Y(x):\,Y\in\mathrm{Lie}(X)\Big\} = n\qquad\text{
      for every $x\in\R^n$}.
    \end{equation}
    
    \item As a consequence of assumption (H.1), and since the $\sigma_i$'s are
    increasingly ordered (see \eqref{eq.defidela}), we derive from
    \cite[Rem.\,1.3.7]{BLUlibro} that $X_1,\ldots,X_m$ are
    \emph{py\-ra\-mid-sha\-ped}. More precisely, if we write (for $i = 1,\ldots,m$)
    $$X_i = \sum_{k = 1}^n a_{k,i}(x)\,\frac{\de }{\de x_k},$$
    we have that $a_{k,i}(x)$ \emph{does not depend} on the variables $x_k,\ldots,x_n$.
    \vspace*{0.05cm}
    
    \item By crucially exploiting the pyramid-shape of the $X_i$'s, 
    it is easy to see that the following {integration-by-parts formula}
    holds: \emph{if $\Omega\subseteq\R^n$ is a bounded o\-pen set, regular for the
    Divergence Theorem ,
  and if $u,v\in C^1(\overline{\Omega})$, then}
    \begin{equation} \label{eq.intbypartsuv}
     \int_{\Omega}u\, X_iv\,\d x
     = \int_{\de\Omega} u\,v\,\langle X_i(x),\nu\rangle\,\d H^{n-1}
     -\int_{\Omega}v\, X_iu\,\d x,
    \end{equation}
    where $\nu$ is the outward normal to $\Omega$. In particular, choosing
    $u\equiv 1$, we get
    \begin{equation} \label{eq.intbypartsv}
     \int_{\Omega}X_iv\,\d x = \int_{\de\Omega}v\,\langle X_i(x),\nu\rangle\,\d H^{n-1}.
    \end{equation}
   \end{enumerate}
  \end{remark}
 \noindent If $\Omega\subseteq\R^n$ is a bounded open set and
 $u\in C^1(\Omega)$, we define the \emph{$X$-gradient of $u$}
 as
 $$\nabla_X u = \big(X_1 u,\ldots,X_m u\big)\in C(\Omega;\R^m);$$
 moreover, if $F = (F_1,\ldots,F_m)\in C^1(\Omega;\mathbb{R}^m)$, 
 we define the \emph{$X$-divergence of $F$} as
 \begin{equation*}
  \textstyle\mathrm{div}_X F=-\sum_{i=1}^m X_i F_i\in C(\Omega).
 \end{equation*}
 \begin{example} \label{exm.VFs}
  Before proceeding,  we present some concrete examples
  of vector fields $X_1,\ldots,X_m$ satisfying
  the `structural' assumptions (H.1)-(H.2).
  \medskip
  
  (1)\,\,In Euclidean space $\R^n$, let $m = n$ and
   $$X_i := \de_{x_i} \qquad (\text{for $i = 1,\ldots,n$}).$$
   Clearly, $X_1,\ldots,X_n$ are homogeneous of degree $1$
   with respect to
   $$\delta_\lambda(x) = (\lambda x_1,\ldots,\lambda x_n) = \lambda x,$$
   so that assumption (H.1) is fulfilled. Moreover, since
   $$[X_i,X_j] = 0\qquad\text{for all $i,j = 1,\ldots,n$},$$
   we have that $\mathrm{Lie}(X) = \mathrm{span}\{\de_{x_1},\ldots,\de_{x_n}\}$,
   and assumption (H.2) trivially holds. We explicitly notice that, in this context, we have
   \begin{itemize}
    \item $\nabla_X u = \nabla u$ for every $u\in C^1(\Omega)$;
    \item $\mathrm{div}_X(F) = -\mathrm{div}(F)$ for every $F = (F_1,\ldots,F_n)\in C^1(\Omega;\R^n)$.
   \end{itemize}
   Moreover, the `sum of squares' naturally associated with $X$ is nothing but
   $$\textstyle\Delta_X = -\sum_{i = 1}^n\de_{x_i}^2 = -\Delta.$$
   
   (2)\,\,Let $k,n_1,n_2\in\N$ be arbitrarily fixed, and let $n := n_1+n_2\geq 2$. We denote a
   generic point $x\in\R^n \equiv \R^{n_1}\times\R^{n_2}$ by 
   $$\text{$x = (y,t)$, where
   $x\in\R^{n_1}$ and $t\in\R^{n_2}$},$$ 
   and we consider the vector fields on $\R^n$ defined as follows:
   \begin{align*}
    & Y_i := \de_{y_i} \qquad (i = 1,\ldots,n_1); \\
    & T_{i,j} := y_i^k\,\de_{t_j}\qquad (\text{$i = 1,\ldots,n_1$ and $j = 1,\ldots,n_2$}).
   \end{align*}
   Then, it is not difficult to recognize that the family
   $$X:= \big\{Y_i,\,T_{i,j}\,:\,\text{$i = 1,\ldots,n_1$ and $j = 1,\ldots,n_2$}\big\}$$
   satisfies both assumptions (H.1) and (H.2). 
   In fact, an easy computation ba\-sed on \eqref{eq.delaYvector} shows that
   the elements of $X$ are homogeneous of degree $1$ with respect to
   $$\delta_\lambda(x)
   = \delta_\lambda(y,t) := (\lambda y,\lambda^{k+1}t),$$
   and thus assumption (H.1) is fulfilled. As regards (H.2), since
   $$\big[\underbrace{Y_i,[Y_i,\cdots [Y_i}_{\text{$k$ times}},T_{i,j}]\cdots]\big]
   = k!\,\de_{t_j},$$
   we derive that $\mathrm{Lie}(X) \supset \mathrm{span}\{\de_{y_i},\,\de_{t_{i,j}}\}$;
   using this inclusion, it is straight\-for\-ward to conclude that
   \eqref{eq.HormanderX} is satisfied.
   We explicitly notice that the `sum of squares' na\-tu\-ral\-ly associated
   with $X$ is the sub-elliptic operator
   $$\textstyle\Delta_X := -\sum_{i = 1}^{n_1}Y_i^2-\sum_{i = 1}^{n_1}\sum_{j = 1}^{n_2}
   T_{i,j}^2
   = -\Delta_y - \big(y_1^{2k}+\cdots+y_{n_1}^{2k}\big)\,\Delta_t.$$
  
  (3)\,\,In Euclidean space $\R^n$, with $n\geq 2$, we consider the vector fields
  $$X_1 := \de_{x_1}, \qquad X_2 := x_1\de_{x_2}+\frac{x_1^2}{2!}\de_{x_3}
  +\cdots+\frac{x_1^{n-1}}{(n-1)!}\,\de_{x_n}.$$
  We claim that the family 
  $$X := \{X_1,X_2\}$$ 
  satisfies both assumptions (H.1)
  and (H.2). In fact, a direct computation based on 
  \eqref{eq.delaYvector} shows that
   $X_1,X_2$ are homogeneous of degree $1$ with respect to
   $$\delta_\lambda(x) := (\lambda x_1,\lambda^2x_2,\ldots,\lambda^nx_n),$$
   and thus assumption (H.1) is fulfilled.
   As for (H.2), since
   $$Y_i := \big[\underbrace{X_1,[X_1,\cdots [X_1}_{\text{$i$ times}},X_2]\cdots]
   = \de_{x_{i+1}}+\sum_{j = 2}^{n-i}\frac{x_1^{j-1}}{(j-1)!}\,\de_{x_{i+j}}\qquad
   (1\leq i\leq n-1),$$
   we have that $Y_1,\ldots,Y_{n-1}\in \mathrm{Lie}(X)$; as a consequence, we get
   $$\{Y(0):\,Y\in\mathrm{Lie}(X)\}\supset\{e_1,\ldots,e_n\},$$
   and this readily proves that \eqref{eq.HormanderX} is satisfied.
   We explicitly notice that the `sum of squares' na\-tu\-ral\-ly associated
   with $X$ is the sub-elliptic operator
   $$\Delta_X := -X_1^2-X_2^2
   = -\de_{x_1}^2+\Big(x_1\de_{x_2}+\cdots+\frac{x_1^{n-1}}{(n-1)!}\de_{x_n}\Big)^2.$$
   This operator was first introduced by Bony is his
   celebrated 1969 paper \cite{Bony}.
  \end{example}
 With all the above assumptions and notation at hand,
 we are ready to introduce the main object of our study.
 Let $\mathcal{D} := \Omega\times\R\times\R^m$ and let
 $$\mathcal{F} = \mathcal{F}(x,z,p):\mathcal{D}\longto\R$$
 satisfy the next two properties:
 \begin{itemize}
  \item[(P.1)] $\mathcal{F}$ is of class $C^1$ on (an open neighborhood
  of) $\overline{\mathcal{D}}$;
  \item[(P.2)] for every $i = 1,\ldots,m$, the function
  $\de_{p_i}\mathcal{F}$ is of class $C^1$ on 
  $\overline{\mathcal{D}}$.
 \end{itemize}
 In this paper, we shall be concerned with the functional
 \begin{equation} \label{eq.defOpForder1}
  F:C^2({\Omega})\longto\R, \qquad
 F(u) := \int_{\Omega}\mathcal{F}\big(x,u(x),\nabla_Xu(x)\big)\,\d x.
 \end{equation}
 \begin{remark} \label{rem.EulerF}
 We explicitly notice that the functional $F$ is intimately
 related with the following \emph{non-linear sub-elliptic} PDE
 \begin{equation} \label{eq.PDEorder1}
  \mathrm{div}_X\Big(\mathcal{F}_p
 \big(x,u(x),\nabla_X u(x)\big)\Big)+\de_z\FF\big(x,u(x),\nabla_X u(x)\big)=0,
 \end{equation}
 where we have used the obvious notation
 $$
 \mathcal{F}_p = 
 \big(\de_{p_1}\FF,\ldots,\de_{p_m}\FF\big).$$
 In fact, if $u\in C^2(\Omega)$ is a
 minimum/maximum point of $F$, we have
 $$\frac{\d}{\d t}\Big|_{t = 0}F(u+t\varphi) = 0\qquad
 \text{for all $\varphi\in C_0^\infty(\Omega)$}.$$
 From this, taking into account properties (P.1)-(P.2) of
 $\mathcal{F}$,
 an application of Le\-be\-sgue's Dominated Convergence Theorem
 shows that, for every $\varphi\in C_0^\infty(\Omega)$,
 \begin{align*}
 & 0 = \int_{\Omega}\Big[\de_z\FF\big(x,u(x),\nabla_Xu(x)\big)\,\varphi
 +\langle\mathcal{F}_p\big(x,u(x),\nabla_Xu(x)\big),\nabla_X\varphi\rangle\Big]\d x
 \\[0.1cm]
 & \qquad (\text{using \eqref{eq.intbypartsuv}, and since $\varphi\in C_0^\infty(\Omega)$}) 
 \\[0.1cm]
 & \quad = 
 \int_{\Omega}\Big[\de_z\FF\big(x,u(x),\nabla_Xu(x)\big)
 +\mathrm{div}_X\Big(\mathcal{F}_p
 \big(x,u(x),\nabla_X u(x)\big)\Big)\Big]\varphi\,\d x,
 \end{align*}
 and thus $u$ satisfies \eqref{eq.PDEorder1} (point-wise on $\Omega$).
 \end{remark}
  \section{Rellich-Pohozaev identity for ${F}$} \label{sec.Orderone}
  Throughout the sequel, $X = \{X_1,\ldots,X_m\}\subseteq\mathcal{X}(\R^n)$
  is a fixed family sa\-ti\-sfy\-ing assumptions (H.1)-(H.2) and $\Omega$
  is a bounded open set, regular for the Divergence Theorem.
  Moreover, $F$ is as in \eqref{eq.defOpForder1} with $\mathcal{F}$ fulfilling
  (P.1)-(P.2).
  \medskip
  
  Our main aim in this section 
  is to prove some integral identities related to $F$.
  To this end, if $\sigma_1,\ldots,\sigma_n$ are as in \eqref{eq.defidela}, 
 we introduce the vector field
 \begin{equation} \label{eq.defiVFT}
  \mathrm{T} := \sum_{i = 1}^n\sigma_ix_i\,\frac{\de }{\de x_i},
 \end{equation}
 which shall be referred to as the \emph{infinitesimal generator 
 of $\{\delta_\lambda\}_{\lambda > 0}$}.
 \begin{remark} \label{rem.relationTdil}
  We explicitly observe, for a future reference, that the vector field 
  $\mathrm{T}$
  is related with the family $\{\delta_\lambda\}_{\lambda > 0}$
  via the following `Euler-type' theorem: \emph{a fun\-ction $u\in C^\infty(\R^n)$
  is $\delta_\lambda$\--ho\-mo\-ge\-neous of degree $m\geq 0$ if and only if}
  $$\mathrm{T}u = mu.$$
  As a consequence, taking into account \eqref{eq.delaYcoeff}, 
  we deduce that a smooth vector field $Y\in\mathcal{X}(\R^n)$ is
  $\delta_\lambda$-homogeneous of degree $\alpha\in\R$ if and only if
  $$[Y,\mathrm{T}] = \alpha Y.$$
 \end{remark}
  We are now ready to prove the following Pohozaev identity for $F$.
  In order to keep our results
  as clear as possible, in the sequel we shall write
  \begin{align*}
   & \mathcal{F},\quad
   \mathcal{F}_x := \big(\de_{x_1}\FF,\ldots,\de_{x_n}\FF\big), 
   \quad \de_z\FF,\quad
  \mathcal{F}_p = \big(\de_{p_1}\FF,\ldots,\de_{p_m}\FF\big)
  \end{align*}
  and, unless otherwise specified, 
  we shall understand that the above functions
  are evaluated at points
  $(x,u(x),\nabla_Xu(x))$ (with $x\in\Omega$).
 \begin{proposition} \label{prop.PohoLL}
  For every $u\in C^2(\overline{\Omega})$, we have the identity
  \begin{equation}\label{eq.PohoLL}
  \begin{split}
   & \int_{\Omega}
   \big(
   q\mathcal{F}-\langle \mathcal{F}_p, \nabla_X u\rangle\big)\,\d x
   +
   \int_{\Omega}\mathrm{T}\big(x\mapsto\mathcal{F}(x,z,p)\big)(x,u,\nabla_Xu)\, \d x 
   \\
   & \qquad\qquad\qquad +
   \int_{\Omega}\mathrm{T}u\,\big(\mathrm{div}_X(\mathcal{F}_p)
   +\de_z\mathcal{F}\big)\, \d x \\
  &\qquad = \int_{\de\Omega}\mathcal{F}\,\langle \mathrm{T}(x),\nu\rangle
  \,\d H^{n-1}- \int_{\de\Omega}\mathrm{T}u\,
  \langle \mathcal{F}_p, \nu_X\rangle\,\d H^{n-1}
   \end{split}
  \end{equation}
  where $H^{n-1}$ is the standard $(n-1)$-dimensional Hausdorff measure
  in $\R^n$, $\nu$ is the outward normal to $\Omega$ and
  $\nu_X$ is given by
  \begin{equation} \label{eq.defnuX}
   \nu_X := \big(\langle X_1(x), \nu\rangle,\ldots, \langle X_m(x),\nu\rangle\big).
   \end{equation}
  Finally, $q$ is the $\delta_\lambda$-homogeneous dimension
  of $\R^n$ defined in \eqref{eq.defq}.
 \end{proposition}
 \begin{proof}
  Since $\Omega$ is regular for the Divergence Theorem,
  we have
  $$\int_{\Omega}\mathrm{div}\big(\mathcal{F}\cdot\mathrm{T}(x)\big)\,\d x
  = \int_{\de\Omega}\mathcal{F}\,\langle \mathrm{T}(x),\nu\rangle\,\d H^{n-1}.$$
  From this, observing that $\mathrm{div}(\mathrm{T}(x)) = \sum_{i = 1}^n\sigma_i = q$
  (see \eqref{eq.defq}),
  we obtain 
  \begin{equation} \label{eq.stepIPoho}
  \begin{split}
   & q\int_{\Omega}\mathcal{F}\,\d x
   + \int_{\Omega}\langle \nabla \mathcal{F}, \Tv(x)\rangle\,\d x
   = \int_{\de\Omega}\mathcal{F}\,\langle \mathrm{T}(x),\nu\rangle\,\d H^{n-1},
  \end{split}
  \end{equation}
  where we have used the notation
  $$\nabla\mathcal{F} = \nabla\big(x\mapsto\mathcal{F}\big(x,u(x),\nabla_Xu(x)\big)\big).$$
  Now, according to this notation, a direct calculation gives
  \begin{equation} \label{eq.nablaFfull}
  \begin{split}
  \langle \nabla \mathcal{F}, \Tv(x)\rangle
  & = \Tv\big(x\mapsto\mathcal{F}(x,z,p)\big)\big(x,u(x),\nabla_X u(x)\big) \\[0.1cm]
  & +\mathcal{F}_u\cdot \Tv u(x)+
  \sum_{i = 1}^m\frac{\de\mathcal{F}}{\de p_i}\cdot 
  \Tv(X_iu).
   \end{split}
  \end{equation}  
  On the other hand, since the $X_i$'s are $\delta_\lambda$-homogeneous of degree
  $1$, from Remark \ref{rem.relationTdil} we derive that
  $[X_i,\mathrm{T}] = X_i$ (for every $i = 1,\ldots, m$); hence,
  we can write
  \begin{equation} \label{eq.FpnablaXucomm}
  \begin{split}
  \textstyle\sum_{i = 1}^m\de_{p_i}\FF\cdot 
  \Tv(X_iu) & = \textstyle\sum_{i = 1}^m\de_{p_i}\FF\cdot 
  X_i(\Tv u) + \textstyle\sum_{i = 1}^m\de_{p_i}\FF\cdot 
  [\Tv, X_i]u \\[0.15cm]
   & 
   = \langle \mathcal{F}_p,\nabla_X(\Tv u)\rangle
   - \textstyle\sum_{i = 1}^m\de_{p_i}\FF\cdot X_iu \\[0.15cm]
   & = \langle \mathcal{F}_p,\nabla_X(\Tv u)\rangle -\langle \mathcal{F}_p,\nabla_Xu\rangle.
   \end{split}
  \end{equation}
  By combining \eqref{eq.stepIPoho}, \eqref{eq.nablaFfull}
  and \eqref{eq.FpnablaXucomm}, we then get
  \begin{equation} \label{eq.StepIIPoho}
   \begin{split}
   & \int_{\Omega}\big(q\mathcal{F}-\langle \mathcal{F}_p, \nabla_X u\rangle\big)\,\d x
   +\int_{\Omega} \big[\Tv\big(x\mapsto\mathcal{F}(x,z,p)\big)(x,u,\nabla_X u)
   +\de_z\mathcal{F}\cdot \Tv u\big]\,\d x \\
   & \qquad\qquad\quad +
   \int_{\Omega} \langle \mathcal{F}_p, \nabla_X(\Tv u)\rangle\,\d x 
    = \int_{\de\Omega}\mathcal{F}\,\langle \mathrm{T}(x),\nu\rangle\,\d H^{n-1}.
   \end{split}
  \end{equation}
  To proceed further, we integrate by parts in the last integral in the
  left-hand side of \eqref{eq.StepIIPoho}: 
  using formula \eqref{eq.intbypartsuv} in Remark \ref{rem.conseqH}-(3), we get
  \begin{equation} \label{eq.afterIntegrPart}
  \begin{split}
   & \int_{\Omega} \de_{p_i}\FF\cdot 
   X_i(\mathrm{T}u)\,\d x
    = \int_{\de\Omega}\de_{p_i}\FF\cdot\mathrm{T}u\cdot \langle 
    X_i(x),\nu\rangle\,\d H^{n-1}
    - \int_{\Omega}\mathrm{T}u\,X_i(\mathcal{F}_{p_i})\,\d x.
  \end{split}
  \end{equation}
  Gathering together \eqref{eq.StepIIPoho}, \eqref{eq.afterIntegrPart} and
  the definition of $\nu_X$ in \eqref{eq.defnuX}, we then obtain
  \begin{align*}
  \begin{split}
   & \int_{\Omega}\big(
   q\mathcal{F}-\langle \nabla_p \mathcal{F}, \nabla_X u\rangle\big)\,\d x
   +
   \int_{\Omega}\Tv\big(x\mapsto\mathcal{F}(x,z,p)\big)(x,u,\nabla_X u)\,\d x 
   \\
   &\qquad\qquad\qquad +\int_{\Omega}\mathrm{T}u\big(\mathrm{div}_X(\mathcal{F}_p)
   +\de_z\FF\big)\,\d x \\[0.1cm]
   & \qquad = \int_{\de\Omega}\mathcal{F}\,\langle \mathrm{T}(x),\nu\rangle\,\d H^{n-1}
   - \int_{\de\Omega}\mathrm{T}u\cdot\langle \FF_p,\nu_X\rangle\,\d H^{n-1},
   \end{split}
  \end{align*}
  which is exactly the desired \eqref{eq.PohoLL}.
 \end{proof}  
 Just about to illustrate the applicability
 of Proposition \ref{prop.PohoLL}, we write down
 the integral identity \eqref{eq.PohoLL} for a general $\FF$
 in two concrete examples.
 \begin{example} \label{exm.classicalLap}
  In Euclidean space $\R^n$, let $m = n$ and
  $$X_1 = \de_{x_1},\ldots, X_n = \de_{x_n}.$$
  Moreover, let $\Omega\subseteq\R^n$ be a fixed open set, regular for the Divergence
  Theorem, and let $k \in (1,n]$ be a fixed real number. We then define
  $$\mathcal{F}:\Omega\times\R\times\R^n\longto\R,\qquad
  \mathcal{F}(x,z,p) := \frac{|p|^k}{k}-G(z),$$
  where $G\in C^1(\R)$ is a suitable function satisfying $G(0) = 0$. \vspace*{0.05cm}
  
  We have already recognized in Example \ref{exm.VFs}-(1) that $X
  = \{X_1,\ldots,X_n\}$
  satisfies assumption (H.1) and (H.2); in particular, we have
  $$\nabla_X u = \nabla u \qquad
  \text{for all $u\in C^1(\Omega)$}.$$
  Moreover, since the $X_i$'s are homogeneous of degree $1$ with respect to
  $$\delta_\lambda(x) = \big(\lambda x_1,\ldots,\lambda x_n\big),$$
  from \eqref{eq.defq} and \eqref{eq.defiVFT} we derive that
  $$q = n\qquad\text{and}\qquad\Tv = \textstyle\sum_{i = 1}^n x_i\,\de_{x_i}.$$
  Taking into account all these facts, 
  identity \eqref{eq.PohoLL} boils down to
  \begin{equation} \label{eq.PohokLapl}
   \begin{split}
   & \bigg(\frac{n}{k}-1\bigg)\int_{\Omega}|\nabla u|^{k}\,\d x
   -n\int_{\Omega}G(u)\,\d x
   + \int_{\Omega}\langle x,\nabla u\rangle\big(\Delta_ku+G'(u)\big)\,\d x \\[0.1cm]
   & 
   \qquad = \frac{1}{k}\int_{\de\Omega}\langle x,\nu\rangle\,|\nabla u|^k\,\d H^{n-1} 
   -\int_{\de\Omega}|\nabla u|^{k-2}\langle x,\nabla u\rangle\,\frac{\de u}{\de \nu}\,\d H^{n-1},
  \end{split}
 \end{equation}
  where $\Delta_k u = -\mathrm{div}(|\nabla u|^{k-2}\,\nabla u)$ is
  the usual $k$-Laplacian of $u$. 
 \end{example}
 \begin{example} \label{exm.LapOriz}
  In Euclidean space $\R^n$, let $m < n$ and let
  $X = \{X_1,\ldots,X_m\}$ satisfy assumptions (H.1)-(H.2).
  Moreover, let $\Omega\subseteq\R^n$ be a fixed open set, regular
  for the Divergence Theorem, and let
  $k\in(1,n]$. As in Example \ref{exm.classicalLap}, we then define
  $$\mathcal{F}:\Omega\times\R\times\R^n\longto\R,\qquad
  \mathcal{F}(x,z,p) := \frac{|p|^k}{k}-G(u),$$
  where $G\in C^1(\R)$ and $G(0) = 0$. In this scenario, 
  identity \eqref{eq.PohoLL} boils down to
  \begin{equation} \label{eq.PohokLaplHor}
   \begin{split}
   & \bigg(\frac{q}{k}-1\bigg)\int_{\Omega}|\nabla_X u|^{k}\,\d x
   -q\int_{\Omega}G(u)\,\d x
   + \int_{\Omega}\Tv u\big(\Delta_{X,k}u+G'(u)\big)\,\d x \\[0.1cm]
   & 
   \qquad = \frac{1}{k}\int_{\de\Omega}\langle \Tv(x),\nu\rangle\,|\nabla_X u|^k\,\d H^{n-1} 
   -\int_{\de\Omega}\Tv u\,|\nabla u|^{k-2}\,\langle\nabla_X u,\nu_X\rangle\,\d H^{n-1},
  \end{split}
 \end{equation}
  where $\Delta_{X,k}u$ is the so-called \emph{horizontal $k$-Laplacian} of $u$, that is, 
  $$\Delta_{X,k}u = \mathrm{div}_X(|\nabla_X u|^{k-2}\,\nabla_X u),$$
  and $q,\Tv,\nu_X$ are given, respectively, by \eqref{eq.defq},
  \eqref{eq.defiVFT}, \eqref{eq.defnuX}.
 \end{example}
 \subsection{Applications}
  As it happens in the `classical Euclidean case', Proposition \ref{prop.PohoLL} is a key 
  tool in the study of the PDE \eqref{eq.PDEorder1} naturally associated with $F$, i.e.,
  \begin{equation} \label{eq.semilinearLL}
   \mathrm{div}_X\big(\mathcal{F}_p\big(x,u(x),\nabla_X u(x)\big)\big)+
   \de_z\mathcal{F}\big(x,u(x),\nabla_X u(x)\big)=0.
  \end{equation}
  More precisely, we have the following key result.
  \begin{theorem} \label{thm.IntegralPDE}
   Let $u\in C^2(\overline{\Omega})$ be a solution
   of \eqref{eq.semilinearLL}. Then
   \begin{equation} \label{eq.PohoPDE}
  \begin{split}
   & \int_{\Omega}
   \big(
   q\mathcal{F}-\langle \mathcal{F}_p, \nabla_X u\rangle\big)\,\d x
   +
   \int_{\Omega}\mathrm{T}\big(x\mapsto\mathcal{F}(x,z,p)\big)(x,u,\nabla_Xu)\, \d x 
   \\
  &\qquad = \int_{\de\Omega}\mathcal{F}\,\langle \mathrm{T}(x),\nu\rangle
  \,\d H^{n-1}- \int_{\de\Omega}\mathrm{T}u\,
  \langle \mathcal{F}_p, \nu_X\rangle\,\d H^{n-1}.
   \end{split}
   \end{equation}
   If, in addition, $u\equiv 0$ on $\de\Omega$, for every $a\in\R$ we have
   \begin{equation} \label{eq.PohoBVPzero}
   \begin{split}
   & \int_{\Omega}
   \big(q\mathcal{F}-(a+1)\langle \mathcal{F}_p, \nabla_X u\rangle\big)\,\d x
   +
   \int_{\Omega}\mathrm{T}\big(x\mapsto\mathcal{F}(x,z,p)\big)(x,u,\nabla_Xu)\,\d x 
   \\
  & \qquad 
  -a\int_{\Omega}u\,\de_z\FF\,\d x = \int_{\de\Omega}\big(\mathcal{F}
  -\langle\FF_p,\nabla_X u\rangle\big)\cdot
  \langle \mathrm{T}(x),\nu\rangle\,\d H^{n-1}.
   \end{split}
   \end{equation}
  \end{theorem}
  \begin{proof}
  As regards identity \eqref{eq.PohoPDE}, it is an
   immediate consequence of \eqref{eq.PohoLL} and of the fact
   that, $u$ being a solution of \eqref{eq.semilinearLL}, one has
   $$\int_{\Omega}\mathrm{T}u\,\big(\mathrm{div}_X(\mathcal{F}_p)
   +\de_z\mathcal{F}\big)\, \d x = 0.$$
   We then turn to prove identity \eqref{eq.PohoBVPzero} under the additional
   assumption that $u\equiv 0$ on $\de\Omega$. To this end,
   we first claim that
   \begin{equation} \label{eq.claimedwitha}
    \int_{\Omega}\big(\langle \mathcal{F}_p, \nabla_X u\rangle+u\,\de_z\FF\big)\,\d x = 0.
   \end{equation}
   Indeed, since $u$ is a solution of \eqref{eq.semilinearLL}, we have
   \begin{align*}
    \textstyle \sum_{i = 1}^mX_i\big(u\,\de_{p_i}\FF\big)
    & = \langle\nabla_X u,\FF_p\rangle - u\,\mathrm{div}_X(\FF_p)
    = \langle \mathcal{F}_p, \nabla_X u\rangle+u\,\de_z\FF;
   \end{align*}
   from this, using Divergence's Theorem and the fact that
   $u\equiv 0$ on $\de\Omega$, we get
   \begin{align*}
    & \int_{\Omega}\big(\langle \mathcal{F}_p, \nabla_X u\rangle+u\,\de_z\FF\big)\,\d x
     = \sum_{i = 1}^m\int_{\Omega}X_i\big(u\,\de_{p_i}\FF\big)\,\d x
    \\[0.1cm]
    & \qquad\quad (\text{writing, as usual, 
    $\textstyle X_i = \sum_{k=1}^na_{k,i}(x)\,\de_{x_k}$}) \\[0.1cm]
    & \qquad\quad 
    = \sum_{i = 1}^m\sum_{k = 1}^n
    \int_{\Omega}a_{k,i}(x)\,\de_{x_k}\big(u\,\de_{p_i}\FF\big)\,\d x
    \\
    & \qquad\quad = 
    \sum_{i = 1}^m\sum_{k = 1}^n
    \bigg(\int_{\de\Omega}a_{k,i}(x)\,u\,\de_{p_i}\FF\cdot\nu_k
    - \int_{\Omega}u\,\de_{p_i}\FF\cdot\de_{x_k}a_{k,i}\,\d x\bigg) 
    = 0,
   \end{align*}
   which is precisely the claimed \eqref{eq.claimedwitha}
   (see also Remark \ref{rem.conseqH}-(2)).
   By combining this last identity with \eqref{eq.PohoPDE}
   we then obtain, for every choice of $a\in\R$,
   \begin{equation} \label{eq.PohoPDEa}
  \begin{split}
   & \int_{\Omega}
   \big(
   q\mathcal{F}-(a+1)\langle \mathcal{F}_p, \nabla_X u\rangle\big)\,\d x
   +
   \int_{\Omega}\mathrm{T}\big(x\mapsto\mathcal{F}(x,z,p)\big)(x,u,\nabla_Xu)\, \d x 
   \\
  & \,\, 
  -a\int_{\Omega}u\,\de_z\FF\,\d x = \int_{\de\Omega}\mathcal{F}\,\langle \mathrm{T}(x),\nu\rangle
  \,\d H^{n-1}- \int_{\de\Omega}\mathrm{T}u\,
  \langle \mathcal{F}_p, \nu_X\rangle\,\d H^{n-1}.
   \end{split}
   \end{equation}
   To complete the demonstration, we now perform some algebraic computations
   on the term $\mathrm{T}u\,
  \langle \mathcal{F}_p, \nu_X\rangle$.
  Using the very definition of $\Tv$ and $\nu_X$, and writing
  $$\textstyle X_k = \sum_{j = 1}^na_{j,k}(x)\,\de_{x_j} \qquad
  (i = 1,\ldots,m),$$
  we derive the following chain of equality:
  \begin{align*}
   \mathrm{T}u\,
  \langle \mathcal{F}_p, \nu_X\rangle
  & = \sum_{i = 1}^n\sum_{k = 1}^m\sigma_i\,x_i\,\de_{x_i}u\cdot
  \de_{p_k}\FF\cdot\langle X_k(x),\nu\rangle \\
  & = \sum_{i,j = 1}^n\sum_{k = 1}^m\sigma_i\,x_i\,\de_{x_i}u\cdot
  \de_{p_k}\FF\cdot a_{j,k}(x)\,\nu_j =: (\bigstar).
  \end{align*}
  On the other hand, since $u\equiv 0$ on $\de\Omega$, for every $i = 1,\ldots, n$ we have
  $$\de_{x_i}u = \frac{\de u}{\de \nu}\cdot\nu_i;
  $$
  as a consequence, we obtain
  \begin{align*}
   (\bigstar) & = 
   \sum_{i,j = 1}^n\sum_{k = 1}^m
   \de_{p_k}\FF\cdot\big(\sigma_i\,x_i\,\nu_i\big)\cdot
  a_{j,k}(x)\,\Big(\frac{\de u}{\de \nu}\cdot\nu_j\Big)  \\
  & = \Big(\sum_{i = 1}^n\sigma_i\,x_i\,\nu_i\Big)\cdot
  \sum_{k = 1}^m\de_{p_k}\FF\cdot
  \Big(\sum_{j = 1}^n a_{j,k}(x)\,\de_{x_j}u\Big) \\
  & = \langle \Tv(x),\nu\rangle\cdot 
  \sum_{k = 1}^m\de_{p_k}\FF\cdot X_ku 
   = \langle \Tv(x),\nu\rangle\cdot \langle\FF_p,\nabla_X u\rangle.
  \end{align*}
  By inserting this last identity into 
  \eqref{eq.PohoPDEa}, we finally get 
  \begin{equation*}
  \begin{split}
   & \int_{\Omega}
   \big(
   q\mathcal{F}-(a+1)\langle \mathcal{F}_p, \nabla_X u\rangle\big)\,\d x
   +
   \int_{\Omega}\mathrm{T}\big(x\mapsto\mathcal{F}(x,z,p)\big)(x,u,\nabla_Xu)\, \d x 
   \\
  & \qquad 
  -a\int_{\Omega}u\,\de_z\FF\,\d x = \int_{\de\Omega}
  \langle \mathrm{T}(x),\nu\rangle\big(\FF-\langle\FF_p,\nabla_X u\rangle\big)\,\d H^{n-1},
   \end{split}
   \end{equation*}
  which is exactly the desired \eqref{eq.PohoBVPzero}. This ends the proof.
  \end{proof}
  With Theorem \ref{thm.IntegralPDE} at hand, we are ready to 
  prove our non-existence result for the Dirichlet
  problem associated with \eqref{eq.semilinearLL}.
  Before doing this, we fix a definition.
  \begin{definition} \label{def.starshaped}
   Let $\Omega\subseteq\R^n$ be a non-void open set with $C^1$ boundary. 
   We say that $\Omega$
   is \emph{$\delta_\lambda$-star shaped} (with respect to the origin) if
   \begin{equation} \label{eq.starshaped}
    \langle\Tv(x),\nu\rangle\geq 0\qquad\text{for every $x\in\de\Omega$},
   \end{equation}
   where $\Tv$ is as in \eqref{eq.defiVFT} and $\nu$ is the outward normal to $\de\Omega$.
  \end{definition}
  \begin{theorem} \label{thm.mainNonexistOrder1}
   Let $\Omega\subseteq\R^n$ be a \emph{connected} open set, regular for the Divergence Theorem
   and $\delta_\lambda$-star shaped with respect to the origin. 
   We assume that
   \begin{itemize}
    \item[(i)] $\FF(x,0,p)-\langle\FF_p(x,0,p),p\rangle\leq 0$ 
    for all $x\in\de\Omega$ and $p\in\R^m$; \vspace*{0.07cm}
    
    \item[(ii)] there exists a number $a_0\in\R$ such that
    \begin{equation} \label{eq.signFFnonexistence}
    \begin{split}
     & q\FF(x,z,p)-(a_0+1)\langle\FF_p(x,z,p),p\rangle \\
     & \qquad +
     \mathrm{T}\big(x\mapsto\mathcal{F}(x,z,p)\big)-a_0\,z\,\de_z\FF(x,z,p)
     \geq 0
     \end{split}
    \end{equation}
    for every $x\in\Omega$ and every $(z,p)\in\R\times\R^m$; \vspace*{0.07cm}
    
    \item[(iii)] either $z = 0$ or $p = 0$ when \eqref{eq.signFFnonexistence}
    holds.
   \end{itemize}
   Then, the boundary-value problem
   \begin{equation} \label{eq.BVPoder1}
    \begin{cases}
    \mathrm{div}_X\big(\mathcal{F}_p\big(x,u(x),\nabla_X u(x)\big)\big)+
   \de_z\mathcal{F}\big(x,u(x),\nabla_X u(x)\big) = 0 & \text{in $\Omega$}, \\[0.1cm]
   u\big|_{\de\Omega} = 0
   \end{cases}
   \end{equation}
   has no non-trivial solutions $u\in C^2(\overline{\Omega})$.
  \end{theorem}
  \begin{proof}
   Let $u\in C^2(\overline{\Omega})$ be a solution
   of \eqref{eq.BVPoder1}. We need to prove that
   \begin{equation} \label{eq.toproveuzero}
   \text{$u\equiv 0$ in $\Omega$}.
   \end{equation}
   First of all, since $u\equiv 0$ on $\de\Omega$ 
   and since $\Omega$ is $\delta_\lambda$-star shaped (with respect to the origin),
   by combining assumption (i) and identity \eqref{eq.PohoPDEa} we derive that
   \begin{equation*}
    \begin{split}
   & \int_{\Omega}
   \big(q\mathcal{F}-(a+1)\langle \mathcal{F}_p, \nabla_X u\rangle\big)\,\d x
   +
   \int_{\Omega}\mathrm{T}\big(x\mapsto\mathcal{F}(x,z,p)\big)(x,u,\nabla_Xu)\,\d x 
   \\
  & \qquad 
  -a\int_{\Omega}u\,\de_z\FF\,\d x \leq 0 \qquad (\text{for every $a\in\R$});
   \end{split}
   \end{equation*}
   from this, using assumption (ii) we obtain
   \begin{equation*}
    \begin{split}
     q\FF - (a_0+1)\langle \mathcal{F}_p, \nabla_X u\rangle 
   + \mathrm{T}\big(x\mapsto\mathcal{F}(x,z,p)\big)(x,u,\nabla_Xu)
   -a_0\,u\,\de_z\FF = 0
   \end{split}
   \end{equation*}
   for every $x\in\Omega$. Now, according to assumption (iii), only
   two cases can occur: \medskip
   
    (a)\,\,$u\equiv 0$ on $\Omega$. In this case, \eqref{eq.toproveuzero} is
    satisfied and the proof is complete. \medskip
    
    (b)\,\,$\nabla_X u \equiv 0$ on $\Omega$. In this case,
    since $X_iu\equiv 0$ on $\Omega$ for every $i = 1,\ldots,m$,
    we see that $u$ \emph{is constant along any integral curve
    of $\pm X_1,\ldots,\pm X_m$ contained in $\Omega$}.
    On the other hand, since $\Omega$ is connected and the $X_i$'s
    satisfy H\"or\-man\-der's rank condition,
    the Chow-Rashevsky Connectivity Theorem holds (see, e.g., \cite[Thm.\,6.22]{BBBook}):
    for e\-ve\-ry $x,y\in\Omega$ there exists a continuous path $\gamma:[0,1]\to\Omega$, which is
    piecewise an integral curve of the vector fields $\pm X_1,\ldots,\pm X_m$, such that
    $$\gamma(0) = x\qquad\text{and}\qquad\gamma(1) = y.$$
    As a consequence, $u$ is constant along $\gamma$ and, in particular,
    $$u(x) = u(y).$$
    Due to the arbitrariness of $x,y\in\Omega$, we then conclude that
    $u$ is constant throughout on $\Omega$. From this, since $u\equiv 0$
    on $\de\Omega$, we derive that $u\equiv 0$ on $\Omega$, and the needed
    \eqref{eq.toproveuzero} 
    is again satisfied. This ends the proof.
  \end{proof}
  We end this section by `specializing' the assumptions of
  Theorem \ref{thm.mainNonexistOrder1} in the concrete cases
  discussed in Examples \ref{exm.classicalLap} and \ref{exm.LapOriz}.
  \begin{example} \label{exm.nonexistclassicalLap}
   In Euclidean space $\R^n$,
   let $m = n$ and let $X = \{X_1,\ldots,X_n\},\,\FF$ be as in Example \ref{exm.classicalLap}.
 	For every $x\in\R^n$ and every $p\in\R^n$, we have
 	\begin{align*}
 	 \FF(x,0,p)-\langle\FF_p(x,0,p),p\rangle
 	 & = \frac{|p|^k}{k} - G(0) - |p|^{k-2}\langle p,p\rangle 
 	 = |p|^k\bigg(\frac{1}{k}-1\bigg)\leq 0
 	\end{align*}
 	(as $k > 1$ and $G(0) = 0$); hence, assumption (i) of Theorem \ref{thm.mainNonexistOrder1} 
 	is satisfied. Moreover, since in this case $q = n$
 	and $\FF$ is independent of $x$, one also has
 	\begin{align*}
 	 & q\FF(x,z,p)-(a+1)\langle\FF_p(x,z,p),p\rangle +
     \mathrm{T}\big(x\mapsto\mathcal{F}(x,z,p)\big)-a\,z\,\de_z\FF(x,z,p) \\[0.1cm]
     & \qquad\qquad = \frac{n}{k}|p|^k-nG(z)-(a+1)|p|^k
     +az\,G'(z).
 	\end{align*}
	Let us now suppose that $G$ enjoys the following
	additional properties:
	\begin{itemize}
	 \item[(a)] $G'(0) = 0$; \vspace*{0.05cm}
	 \item[(b)] for every $z\in\R\setminus\{0\}$ we have the growth condition
	 \begin{equation} \label{eq.growthG}
	 G(z) < \varrho z\,G'(z), \qquad 
	 \text{where $\varrho := \frac{1}{k}-\frac{1}{n}$}.
	\end{equation}
	\end{itemize}
	Then, it is immediate to check that assumptions (ii)-(iii) of Theorem
	\ref{thm.mainNonexistOrder1} are satisfied with the choice
	$a_0 := \varrho$. As a consequence, the Dirichlet problem
	$$
	 \begin{cases}
	  \Delta_k u + G'(u) = 0 & \text{in $\Omega$}, \\[0.1cm]
	   u\big|_{\de\Omega} = 0,
	 \end{cases}
	$$
	possesses no non-trivial solutions $u\in C^2(\overline{\Omega})$
	if \eqref{eq.growthG} holds.
 \end{example}
  \begin{example} \label{exm.nonexisthorizontalLap}
   In Euclidean space $\R^n$,
   let $m < n$ and let $X = \{X_1,\ldots,X_m\},\,\FF$ be as in Example 
   \ref{exm.LapOriz}. By arguing exactly 
   as in Example \ref{exm.nonexistclassicalLap}, we see that
   assumption (i) of Theorem \ref{thm.mainNonexistOrder1} is satisfied;
   moreover, since $\FF$ is independent of $x$, we have
 	\begin{align*}
 	 & q\FF(x,z,p)-(a+1)\langle\FF_p(x,z,p),p\rangle +
     \mathrm{T}\big(x\mapsto\mathcal{F}(x,z,p)\big)-a\,z\,\de_z\FF(x,z,p) \\[0.1cm]
     & \qquad\qquad = \frac{q}{k}|p|^k-qG(z)-(a+1)|p|^k
     +az\,G'(z),
 	\end{align*}
 	where $q\geq n$ is as in \eqref{eq.defq}. 
 	Again as in Example \ref{exm.nonexistclassicalLap}, let us assume that
 	\begin{itemize}
	 \item[(a)] $G'(0) = 0$; \vspace*{0.05cm}
	 \item[(b)] for every $z\in\R\setminus\{0\}$ we have the growth condition
	 \begin{equation} \label{eq.growthGhor}
	 G(z) < \varrho_q z\,G'(z), \qquad 
	 \text{where $\varrho_q := \frac{1}{k}-\frac{1}{q}$}.
	\end{equation}
	\end{itemize}
	Then, it is immediate to check that assumptions (ii)-(iii) of Theorem
	\ref{thm.mainNonexistOrder1} are satisfied with the choice
	$a_0 := \varrho_q$. As a consequence, the Dirichlet problem
	$$
	 \begin{cases}
	  \Delta_{X,k} u + G'(u) = 0 & \text{in $\Omega$}, \\[0.1cm]
	   u\big|_{\de\Omega} = 0,
	 \end{cases}
	$$
	possesses no non-trivial solutions $u\in C^2(\overline{\Omega})$
	if \eqref{eq.growthGhor} holds.
 \end{example}
 \section{Toward a non-existence result for functionals \\ with $X$-Hessian dependence}
 \label{sec.secondorder}
 In this last section, we discuss a possible extension of the non-existence
 result in Theorem \ref{thm.mainNonexistOrder1} to functionals $F$
 with {$X$-Hessian dependence}. 
 In order the clearly describe this setting, we fix a family
 $X = \{X_1,\ldots,X_m\}\subseteq\mathcal{X}(\R^n)$ satisfying
 as\-sump\-tions (H.1)-(H.2), and we inherit all the notation introduced so far.
 In addition, if $\Omega\subseteq\R^n$ is an open set and if
 $u\in C^2(\Omega)$, we define 
 \begin{equation} \label{eq.defiHHXu}
 \HH_Xu = \begin{pmatrix}
  X_j(X_iu)
 \end{pmatrix}_{i,j = 1}^m
 = \begin{pmatrix}
 X_1^2 u & X_2(X_1u) & \cdots & X_m(X_1u) \\[0.1cm]
 X_1(X_2u) & X_2^2u & \cdots & X_m(X_2u) \\[0.1cm]
 \vdots & \vdots & \ddots & \vdots \\[0.1cm]
 X_1(X_mu) & X_2(X_mu) & \cdots & X_m^2u
 \end{pmatrix},
 \end{equation}
 and we call the matrix $\HH_Xu$ the \emph{$X$-Hessian of $u$}.
 \begin{remark} \label{rem.HXnonsym}
  It should be explicitly noticed that, since the $X_i$'s \emph{do not commute}
  (in general),
  the matrix $\HH_Xu$ \emph{is not symmetric}, i.e.,
  $X_i(X_ju)\not\equiv X_j(X_iu)$.
 \end{remark}
 Let now $\mathcal{O}:=\Omega\times\R\times\R^n\times\R^{m^2}$ and let
 $$\FF:\mathcal{O}\longto\R, \qquad
 \FF = \FF(x,z,p,r) = \FF\big(x,z,p,\{r_{ij}\}_{i,j = 1}^m\big),$$
 satisfy the next three properties:
 \begin{itemize}
  \item[(P.1)] $\FF$ is of class $C^1$ on (an open neighborhood of) $\overline{\mathcal{O}}$;
  \item[(P.2)] for every $i = 1,\ldots,m$, the function
  $\FF_{p_i}$ is of class $C^1$ on $\overline{\mathcal{O}}$;
  \item[(P.3)] for every $i,j = 1,\ldots,m$, 
  the function $\FF_{r_{ij}}$ is of class $C^2$ on $\overline{\mathcal{O}}$.
 \end{itemize}
 We then consider the functional $F:C^4(\Omega)\to\R$ defined as follows:
 \begin{equation} \label{eq.functionalFsecond}
 \begin{split}
  F(u) & := \int_{\Omega}\FF\big(x,u(x),\nabla_Xu(x),\HH_Xu(x)\big)\,\d x \\[0.1cm]
  & = \int_{\Omega}\FF\big(x,u(x),\nabla_Xu(x),\{X_j(X_iu)\}_{i,j = 1}^m\big)\,\d x.
  \end{split}
 \end{equation}
 \begin{remark} \label{rem.EulerLagSecondOrder}
  Analogously to what described in Remark \ref{rem.EulerF}, the functional
  $F$ defined in \eqref{eq.functionalFsecond} is deeply related with
  the following non-linear PDE
  \begin{equation} \label{eq.PDEsecondorder}
  \begin{split}
   \sum_{i,j=1}^m X_iX_j&\Big(\mathcal{F}_{r_{ij}}\big(x,u,\nabla_X u, \HH_X u)\Big)
   +\mathrm{div}_X\Big(\mathcal{F}_p
 \big(x,u(x),\nabla_X u(x),\HH_Xu\big)\Big) \\
 & \qquad\qquad +
 \mathcal{F}_z\big(x,u,\nabla_X u,\HH_X u\big) = 0.
 \end{split}
  \end{equation}
  In fact, if $u\in C^4(\Omega)$ is a
 minimum/maximum point of $F$, we have
 $$\frac{\d}{\d t}\Big|_{t = 0}F(u+t\varphi) = 0\qquad
 \text{for all $\varphi\in C_0^\infty(\Omega)$}.$$
 From this, taking into account properties (P.1)-to-(P.3) of
 $\mathcal{F}$,
 an application of Le\-be\-sgue's Dominated Convergence Theorem
 shows that, for every $\varphi\in C_0^\infty(\Omega)$,
 \begin{align*}
 & 0 = \int_{\Omega}\Big[
 \de_z\FF\big(x,u(x),\nabla_Xu(x),\HH_Xu(x)\big)\,\varphi 
 + \langle\mathcal{F}_p\big(x,u(x),\nabla_Xu(x),\HH_Xu(x)\big),\nabla_X\varphi\rangle \\
 & \qquad\qquad\qquad
 +\textstyle\sum_{i,j = 1}^m\de_{r_{ij}}\FF\big(x,u(x),\nabla_Xu(x),\HH_Xu(x)\big)
 \cdot X_j(X_i\varphi)\Big]\d x
 \\[0.1cm]
 & \qquad (\text{using formula \eqref{eq.intbypartsuv}, and since
 $\varphi\in C_0^\infty(\Omega)$}) \\[0.1cm]
 & \quad
 = \int_{\Omega}\Big[
 \de_z\FF\big(x,u(x),\nabla_Xu(x),\HH_Xu(x)\big) + 
 \mathrm{div}_X\Big(\mathcal{F}_p
 \big(x,u(x),\nabla_X u(x),\HH_Xu(x)\big)\Big) \\
 & \qquad\qquad\qquad + 
 \textstyle\sum_{i,j=1}^m X_iX_j\Big(\mathcal{F}_{r_{ij}}\big(x,u,\nabla_X u, \HH_X u)\Big)
 \Big]\,\varphi\,\d x,
 \end{align*}
 and thus $u$ satisfies \eqref{eq.PDEsecondorder} (point-wise on $\Omega$).
 \end{remark}
 With all these preliminaries at hand, 
 we are ready to prove the analog of Pro\-po\-si\-tion \ref{prop.PohoLL} in this context.
 For notational simplicity, in what follows we write
 \begin{align*}
  & \,\,\,\,\mathcal{F},\quad
   \mathcal{F}_x = \big(\de_{x_1}\FF,\ldots,\de_{x_n}\FF\big), 
   \quad \de_z\FF, \\
   &
  \mathcal{F}_p = \big(\de_{p_1}\FF,\ldots,\de_{p_m}\FF\big),\quad
  \mathcal{F}_{r_{ij}} := \de_{r_{ij}}\FF
 \end{align*}
 and, unless otherwise specified, we understand that all the above functions
 are e\-va\-lu\-a\-ted at points $(x,u(x),\nabla_Xu(x),\HH_Xu(x))$ 
 (for $x\in\Omega$). Moreover,
 we  tacitly assume that the open
 set $\Omega$ is \emph{bounded and regular for the Divergence Theorem}.
 \begin{proposition}
 \label{prop.PohoOderdue}
 For every $u\in C^4(\overline{\Omega})$, we have the identity
 \begin{equation} \label{eq.PohoLLOrderdue}
  \begin{split}
   & \int_{\Omega}\big(q\mathcal{F}-\langle \nabla_p \mathcal{F}, \nabla_X u\rangle\big)\,\d x
   +
   \int_{\Omega}\mathrm{T}\big(x\mapsto\mathcal{F}(x,z,p,r)\big)(x,u,\nabla_Xu,\HH_X u)\, \d x 
   \\[0.1cm]
   & \qquad\quad +
   \int_{\Omega}\mathrm{T}u\,\big(\mathrm{div}_X(\mathcal{F}_p)
   + \textstyle\sum_{i,j=1}^m X_i\big(X_j(\mathcal{F}_{r_{ij}})\big)+\mathcal{F}_u\big)\,\d x
   \\[0.1cm]
   & \qquad\quad - 2\sum_{i,j = 1}^m\int_{\Omega}\mathcal{F}_{r_{ij}}\,X_j(X_i u)\, \d x
   \\[0.1cm]
   & \quad = 
   \int_{\de\Omega}\mathcal{F}\,\langle \mathrm{T}(x),\nu\rangle\,\d H^{n-1} - 
   \int_{\de\Omega}\mathrm{T}u\,\langle \FF_p, \nu_X\rangle\,\d H^{n-1} \\[0.1cm]
   & \qquad\quad
   +\sum_{i,j = 1}^m 
   \int_{\de \Omega}\big(X_i(\mathcal{F}_{r_{ji}})\,\mathrm{T}u\
   -
   \mathcal{F}_{r_{ij}}\,X_i(\mathrm{T}u)\big)
   \langle X_j(x), \nu\rangle\,\d H^{n-1},
   \end{split}
 \end{equation}
 where $\nu$ is the outward normal to $\Omega$ and $\nu_X$ is as in \eqref{eq.defnuX}.
 \end{proposition}
 \begin{proof}
  We argue essentially as in the proof of Proposition \ref{prop.PohoLL}.
  First of all, since $\Omega$ is regular for the Divergence
  Theorem and $\mathrm{div}(\Tv(x)) = q$, we can write
  \begin{equation} \label{eq.stepIPohoOrderDue}
  \begin{split}
   & q\int_{\Omega}\mathcal{F}\,\d x
   + \int_{\Omega}\langle \nabla \mathcal{F}, \Tv(x)\rangle\,\d x
   = \int_{\de\Omega}\mathcal{F}\,\langle \mathrm{T}(x),\nu\rangle\,\d H^{n-1},
  \end{split}
  \end{equation}
  where we have used the notation
  $$\nabla\mathcal{F} = 
  \nabla\big(x\mapsto\mathcal{F}\big(x,u(x),\nabla_Xu(x), \HH_Xu(x)\big)\big).$$
  On the other hand, using identity \eqref{eq.FpnablaXucomm}
  in the proof of Proposition \ref{prop.PohoLL}, we get
  \begin{equation} \label{eq.FpnablaXustepI}
  \begin{split}
  \langle \nabla \mathcal{F}, \Tv(x)\rangle
  & = \Tv\big(x\mapsto\mathcal{F}(x,z,p,r)\big)\big(x,u(x),\nabla_X u(x),\HH_Xu(x)\big) \\[0.1cm]
  & +\mathcal{F}_u\cdot \Tv u(x)+
  \langle \mathcal{F}_p,\nabla_X(\Tv u)\rangle -
  \langle \mathcal{F}_p,\nabla_Xu\rangle \\[0.1cm]
  & + \textstyle\sum_{i,j = 1}^m
  \FF_{r_{ij}}\cdot \Tv\big(X_j(X_iu)\big).
   \end{split}
   \end{equation}
   Now, reminding that $[X_i,\Tv] = X_i$ for every $i = 1,\ldots,m$ (as
   $X_1,\ldots,X_m$ are $\delta_\lambda$-ho\-mo\-ge\-neo\-us of degree $1$,
   see Remark \ref{rem.relationTdil}), we have
   \begin{equation} \label{eq.Frijcommutator}
   \begin{split}
    \FF_{r_{ij}}\cdot \Tv\big(X_j(X_iu)\big) & = 
    \FF_{r_{ij}}\cdot\Big(X_j\big(\Tv(X_iu)\big)
    +[\Tv,X_j](X_iu)\Big) \\
    & = \FF_{r_{ij}}\cdot\Big(X_j\big(X_i(\Tv u)
    + [\Tv, X_i]u\big)
    -X_j(X_iu)\Big)  \\
    & = \FF_{r_{ij}}\cdot\Big(X_j\big(X_i(\Tv u)\big)
    - 2X_j(X_iu)\Big);
   \end{split}
   \end{equation}
   as a consequence, by combining 
   \eqref{eq.stepIPohoOrderDue}, \eqref{eq.FpnablaXustepI} and 
   \eqref{eq.Frijcommutator} we obtain
 \begin{equation} \label{eq.PohoLLOrderdueSemiFinal}
  \begin{split}
   & \int_{\Omega}\big(q\mathcal{F}-\langle \nabla_p \mathcal{F}, \nabla_X u\rangle\big)\,\d x
   +
   \int_{\Omega}\mathrm{T}\big(x\mapsto\mathcal{F}(x,z,p,r)\big)(x,u,\nabla_Xu,\HH_X u)\, \d x 
   \\
   & \quad +
   \int_{\Omega}\Tv u\cdot\FF_u\,\d x
   + \int_{\Omega}\big[\langle \mathcal{F}_p,\nabla_X(\Tv u)\rangle
   + \textstyle\sum_{i,j=1}^m 
   \FF_{r_{ij}}\cdot X_j\big(X_i(\Tv u)\big)\big]\,\d x
   \\
   & \quad - 2\sum_{i,j = 1}^m\int_{\Omega}\mathcal{F}_{r_{ij}}\,X_j(X_i u)\, \d x
   =
   \int_{\de\Omega}\mathcal{F}\,\langle \mathrm{T}(x),\nu\rangle\,\d H^{n-1}.
   \end{split}
 \end{equation}
 To proceed further, we integrate by parts 
 in the fourth integral in the left-hand side
 of \eqref{eq.PohoLLOrderdueSemiFinal}.
 First, using identity
 \eqref{eq.afterIntegrPart} in the proof of Proposition \ref{prop.PohoLL}, we get
 \begin{equation} \label{eq.byPartsPropPrecedente}
 \int_{\Omega}
 \langle \mathcal{F}_p,\nabla_X(\Tv u)\rangle\,\d x
 = \int_{\de\Omega}\Tv u\,\langle\FF_p,\nu_X\rangle\,\d H^{n-1}
 + \int_{\Omega}\Tv u\cdot\mathrm{div}_X(\FF_p)\,\d x;
 \end{equation}
 moreover, by repeatedly exploiting formula
 \eqref{eq.intbypartsuv} in Remark \ref{rem.conseqH}-(3), we have
 \begin{equation} \label{eq.byPartsNuovo}
 \begin{split}
  & \int_{\Omega}\FF_{r_{ij}}\cdot X_j\big(X_i(\Tv u)\big)\,\d x \\
  & \quad = 
  \int_{\de \Omega} \mathcal{F}_{r_{ij}}\cdot X_i(\Tv u)\,
  \langle X_j(x), \nu\rangle\,\d H^{n-1}-
  \int_{\Omega}X_j(\FF_{r_{ij}})\,X_i(\Tv u)\,\d x
  \\
  & \quad =  \int_{\de \Omega} \mathcal{F}_{r_{ij}}\cdot X_i(\Tv u)\,
  \langle X_j(x), \nu\rangle\,\d H^{n-1} \\
  & \qquad -
  \bigg(
  \int_{\de\Omega}
  X_j(\FF_{r_{ij}})\,\Tv u\,\langle X_i(x),\nu\rangle\,\d H^{n-1}
  - \int_{\Omega}\Tv u\cdot X_i\big(X_j(\FF_{r_{ij}})\big)\,\d x
  \bigg).
 \end{split}
 \end{equation}
 By inserting \eqref{eq.byPartsPropPrecedente} and
 \eqref{eq.byPartsNuovo} into 
 \eqref{eq.PohoLLOrderdueSemiFinal}, we finally obtain
 \eqref{eq.PohoLLOrderdue}.
  \end{proof}
  \begin{example} \label{exm.orderdueLapl}
   In Euclidean space $\R^n$, let $m < n$ and let $X = \{X_1,\ldots,X_m\}$
   satisfy assumptions (H.1)-(H.2).
   Moreover, let $\Omega\subseteq\R^n$ be a fixed open set, regular
  for the Divergence Theorem. We then consider the function
  $$\FF:\Omega\times\R\times\R^n\times\R^{m^2}\to\R,
  \qquad
  \FF(x,z,p,r) := \frac{1}{2}\big(\textstyle\sum_{i = 1}^mr_{ii}\big)^2
  - G(z),$$
  where $G\in C^1(\R)$ and $G(0) = 0$. In this context, identity
  \eqref{eq.PohoLLOrderdue} boils down to
  \begin{equation*}
   \begin{split}
   & \Big(\frac{q}{2}-2\Big)\int_{\Omega}(\Delta_Xu)^2\,\d x
   -q\,\int_{\Omega}G(u)\,\d x
   +
   \int_{\Omega}\mathrm{T}u\,\big(
   \Delta_X^2u-G'(u)\big)\,\d x
   \\[0.1cm]
   & \quad = 
   \int_{\de\Omega}\mathcal{F}\,\langle \mathrm{T}(x),\nu\rangle\,\d H^{n-1} 
   \\
   & \quad\qquad +\sum_{i = 1}^m 
   \int_{\de \Omega}\big(X_i(\Delta_Xu)\,\mathrm{T}u\
   -
   \Delta_Xu\,X_i(\mathrm{T}u)\big)
   \langle X_i(x), \nu\rangle\,\d H^{n-1},
   \end{split}
  \end{equation*}
  where $\Delta_{X}u$ is the \emph{horizontal Laplacian} of $u$, that is, 
  $$\Delta_{X}u = -\textstyle\sum_{i = 1}^mX_i^2u,$$
  and $q,\Tv,\nu_X$ are given, respectively, by \eqref{eq.defq},
  \eqref{eq.defiVFT}, \eqref{eq.defnuX}.
  \end{example}
  With Proposition \ref{prop.PohoOderdue} at hand, we could try
  to rerun the arguments of the previous section in order
  to establish a non-existence result for the PDE
  \eqref{eq.PDEsecondorder}, i.e.,
  \begin{equation} \label{eq.PDEorderduevicina}
  \begin{split}
  \sum_{i,j=1}^m X_iX_j&\Big(\mathcal{F}_{r_{ij}}\big(x,u,\nabla_X u, \HH_X u)\Big)
   +\mathrm{div}_X\Big(\mathcal{F}_p
 \big(x,u(x),\nabla_X u(x),\HH_Xu\big)\Big) \\
 & \qquad\qquad +
 \mathcal{F}_z\big(x,u,\nabla_X u,\HH_X u\big) = 0.
  \end{split}
  \end{equation}
  In fact, if $u\in C^4(\overline{\Omega})$ solves \eqref{eq.PDEorderduevicina},
  from identity \eqref{eq.PohoLLOrderdue} we get
  \begin{align*}
  & \int_{\Omega}\big(q\mathcal{F}-\langle \nabla_p \mathcal{F}, \nabla_X u\rangle\big)\,\d x
   +
   \int_{\Omega}\mathrm{T}\big(x\mapsto\mathcal{F}(x,z,p,r)\big)(x,u,\nabla_Xu,\HH_X u)\, \d x 
   \\[0.1cm]
   & \qquad\quad - 2\sum_{i,j = 1}^m\int_{\Omega}\mathcal{F}_{r_{ij}}\,X_j(X_i u)\, \d x
   \\[0.1cm]
   & \quad = 
   \int_{\de\Omega}\mathcal{F}\,\langle \mathrm{T}(x),\nu\rangle\,\d H^{n-1} - 
   \int_{\de\Omega}\mathrm{T}u\,\langle \FF_p, \nu_X\rangle\,\d H^{n-1} \\[0.1cm]
   & \qquad\quad
   +\sum_{i,j = 1}^m 
   \int_{\de \Omega}\big(X_i(\mathcal{F}_{r_{ji}})\,\mathrm{T}u\
   -
   \mathcal{F}_{r_{ij}}\,X_i(\mathrm{T}u)\big)
   \langle X_j(x), \nu\rangle\,\d H^{n-1}.
  \end{align*}
  Moreover, if we further assume that $u\equiv 0$ on $\de\Omega$,
  we already know from the proof of Theorem \ref{thm.IntegralPDE} that,
  for every $x\in\de\Omega$,
  the following equality holds:
  \begin{equation} \label{eq.orineunosulbordo}
   \Tv u\,\langle \mathcal{F}_p, \nu_X\rangle = 
  \langle \Tv(x),\nu\rangle\cdot \langle\FF_p,\nabla_X u\rangle.
  \end{equation}
  Gathering together these facts, we then obtain the identity
  \begin{equation} \label{eq.intermediatenonexistenceOderdue}
   \begin{split}
   & \int_{\Omega}\big(q\mathcal{F}-\langle \nabla_p \mathcal{F}, \nabla_X u\rangle\big)\,\d x
   +
   \int_{\Omega}\mathrm{T}\big(x\mapsto\mathcal{F}(x,z,p,r)\big)(x,u,\nabla_Xu,\HH_X u)\, \d x 
   \\[0.1cm]
   & \qquad\quad - 2\sum_{i,j = 1}^m\int_{\Omega}\mathcal{F}_{r_{ij}}\,X_j(X_i u)\, \d x
   \\[0.1cm]
   & \quad = 
   \int_{\de\Omega}\langle \mathrm{T}(x),\nu\rangle\,
   \big(\mathcal{F}-\langle\FF_p,\nabla_X u\rangle\big)\,\d H^{n-1} \\[0.1cm]
   & \qquad\quad
   +\sum_{i,j = 1}^m 
   \int_{\de \Omega}\big(X_i(\mathcal{F}_{r_{ji}})\,\mathrm{T}u\
   -
   \mathcal{F}_{r_{ij}}\,X_i(\mathrm{T}u)\big)
   \langle X_i(x), \nu\rangle\,\d H^{n-1}.
   \end{split}
  \end{equation}
  Now, following the arguments in the previous section, it is clear that
  identity \eqref{eq.intermediatenonexistenceOderdue} leads to a non-existence
  result for \eqref{eq.PDEorderduevicina} (which a fourth-order PDE)
  if we couple the PDE with \emph{two} boundary conditions, namely
  \begin{equation} \label{eq.boundaryOderdue}
   \text{$u\equiv 0$ on $\de\Omega$}\qquad\text{and}\qquad
  \text{$\mathcal{B}u \equiv 0$ on $\de\Omega$},
  \end{equation}
  where $\mathcal{B}$ is a suitable first-order operator.
  On the other hand, for the `geometrical' assumption
  that $\Omega$ is $\delta_\lambda$-star shaped to play a r\^ole,
  we need to choose $\mathcal{B}$ in such a way that, for every
  solution $u\in C^4(\overline{\Omega})$ satisfying \eqref{eq.boundaryOderdue},
  one has
  \begin{equation} \label{eq.conditionboundaryB}
  \begin{split}
   & \sum_{i,j = 1}^m\big(X_i(\mathcal{F}_{r_{ji}})\,\mathrm{T}u\
   -
   \mathcal{F}_{r_{ij}}\,X_i(\mathrm{T}u)\big)
   \langle X_i(x), \nu\rangle \\
  & \qquad
   = \langle \mathrm{T}(x),\nu\rangle f\big(x,u,\nabla_Xu,\HH_Xu\big),
   \end{split}
  \end{equation}
   where $f = f(x,z,p,r)$ is a suitable function only depending on $\FF$
   and on its de\-ri\-va\-ti\-ves.
   We explicitly point out that, when $\FF$ does not depend on
   $r$, the analog of \eqref{eq.conditionboundaryB} is 
   identity \eqref{eq.orineunosulbordo} (holding true
   under the condition $u\equiv 0$ on $\de\Omega$), where
   $$f(x,z,p) = \langle p,\FF_p(x,z,p)\rangle.$$
   Quite unnaturally, we are able to prove an identity like  
   \eqref{eq.conditionboundaryB} if we assume that
   $\mathcal{B} = \nabla$, that is, if we consider the
   boundary condition
   \begin{equation} \label{eq.boundarynonintrinseca}
    \text{$\nabla u\equiv 0$ on $\de\Omega$}\,\,\Longleftrightarrow\,\,
   \text{$\de_{x_i}u\equiv 0$ on $\de\Omega$ for all $i = 1,\ldots,n$}.
   \end{equation}
   The reason why such a boundary condition is `unnatural' is that
   it is not intrinsically defined trough the vector fields $X_1,\ldots,X_m$ associated
   with $\FF$; in this spirit, a more `natural' condition should involve
   the $X$-gradient of $u$, that is,
   $$\text{$\nabla_Xu \equiv 0$ on $\de\Omega$}.$$
   We believe that the problem of finding adequate \emph{intrinsic} boundary
   conditions which lead
   to non-existence results involving the same notion of $\delta_\lambda$-star
   shapedness used in the previous section could be an interesting challenge.
   \medskip
   
   For the sake of completeness, we end this section by proving the non-existence
   result arising from Proposition \ref{prop.PohoOderdue} under the
   `non-intrinsic' condition \eqref{eq.boundarynonintrinseca}.
    \begin{theorem} \label{thm.nonexistenceOrderdue}
     Let $\Omega\subseteq\R^n$ be a \emph{connected} open set, regular for the Divergence Theorem
   and $\delta_\lambda$-star shaped with respect to the origin. 
   We assume that
   \begin{itemize}
    \item[(i)] 
    for every $x\in\de\Omega$ and every $r\in\R^{m^2}$ we have
    \begin{equation} \label{eq.signleqzeroOrderdue}
    \FF(x,0,0,r)-\textstyle\sum_{i,j = 1}^mr_{ij}\,\FF_{r_{ij}}(x,0,0,r)\leq 0;
    \end{equation} 
    \vspace*{0.07cm}
    
    \item[(ii)] for every $x\in\Omega$ and every $(z,p,r)\in\R\times\R^m\times\R^{m^2}$
    we have
    \begin{equation} \label{eq.signFFnonexistenceOrderdue}
    \begin{split}
     & q\FF(x,z,p,r) - \langle p,\FF_p(x,z,p,r)\rangle  \\
     & \qquad +\mathrm{T}\big(x\mapsto\mathcal{F}(x,z,p,r)\big) - 2
     \textstyle\sum_{i,j = 1}^mr_{ij}\FF_{r_{ij}}(x,z,p,r)
     \geq 0
     \end{split}
    \end{equation}
    \vspace*{0.07cm}
    
    \item[(iii)] either $z = 0$ or $p = 0$ or $r = 0$ when \eqref{eq.signFFnonexistenceOrderdue}
    holds.
   \end{itemize}
   Then, the boundary-value problem
   \begin{equation} \label{eq.BVPoder2}
    \begin{cases}
    \Big(\textstyle\sum_{i,j=1}^m X_iX_j(\mathcal{F}_{r_{ij}})
   +\mathrm{div}_X(\mathcal{F}_p)+
 \mathcal{F}_z\Big)\big(x,u,\nabla_X u,\HH_X u\big) = 0 & \text{in $\Omega$}, \\[0.15cm]
   u\big|_{\de\Omega} = 0, \\[0.15cm]
   \nabla u\big|_{\de\Omega} = 0,
   \end{cases}
   \end{equation}
   has no non-trivial solutions $u\in C^4(\overline{\Omega})$. 
    \end{theorem}
   \begin{proof}
    Let $u\in C^4(\overline{\Omega})$ be a solution of 
    \eqref{eq.BVPoder2}. We need to prove that
    \begin{equation} \label{eq.toproveOrderdue}
     \text{$u\equiv 0$ on $\Omega$}.
    \end{equation}
    We first observe that, since $\nabla u\equiv 0$ on $\de\Omega$, we obviously have
    $$\nabla_X u\equiv 0\quad\text{and}\quad \Tv u = \langle \Tv(x),\nabla u\rangle \equiv 0
    \qquad\text{on $\de\Omega$}.$$
    As a consequence, since
    $u\equiv 0$ on $\de\Omega$, from identity \eqref{eq.intermediatenonexistenceOderdue} 
    we get
    \begin{equation} \label{eq.tosubstituteboundary}
     \begin{split}
     & \int_{\Omega}\big(q\mathcal{F}-\langle \nabla_p \mathcal{F}, \nabla_X u\rangle\big)\,\d x
   +
   \int_{\Omega}\mathrm{T}\big(x\mapsto\mathcal{F}(x,z,p,r)\big)(x,u,\nabla_Xu,\HH_X u)\, \d x 
   \\[0.1cm]
   & \qquad\quad - 2\sum_{i,j = 1}^m\int_{\Omega}\mathcal{F}_{r_{ij}}\,X_j(X_i u)\, \d x
   \\[0.1cm]
   & \quad = 
   \int_{\de\Omega}\langle \mathrm{T}(x),\nu\rangle\,
   \mathcal{F}\,\d H^{n-1} 
   -\sum_{i,j = 1}^m 
   \mathcal{F}_{r_{ij}}X_i(\mathrm{T}u)
   \langle X_j(x), \nu\rangle\,\d H^{n-1}.
     \end{split}
    \end{equation}
    We now claim that, for every $x\in\de\Omega$ and every $i,j = 1,\ldots,m$, one has
    \begin{equation} \label{eq.toproveBoundaryOrderdue}
     X_i(\mathrm{T}u)
   \langle X_j(x), \nu\rangle = \langle \Tv (x),\nu\rangle\cdot X_j(X_iu).
    \end{equation}
    In fact, using the very definition of $\Tv$, and writing
    $$\textstyle X_t = \sum_{h = 1}^na_{h,t}(x)\de_{x_h}\qquad
    (\text{for all $t = 1,\ldots,m$}),$$
    we obtain the following chain of equality
    (holding true for all $x\in\de\Omega$):
    \begin{align*}
     & X_i(\mathrm{T}u)
   \langle X_j(x), \nu\rangle
    = 
   \sum_{h,k = 1}^n a_{h,i}(x)\,\de_{x_h}(\Tv u)\,a_{k,j}(x)\cdot\nu_k
   \\
   & \quad = \sum_{h,k,l = 1}^na_{h,i}(x)\,\de_{x_h}
   \big(\sigma_l\,x_l\,\de_{x_l}u\big)\,a_{k,j}(x)\cdot\nu_k
   \\
   & \quad
   = \sum_{h,k = 1}^n \sigma_h\,a_{h,i}(x)\,(\de_{x_h}u)\,a_{k,j}(x)\cdot\nu_k
   +\sum_{h,k,l = 1}^n
   \sigma_l\,x_l\,a_{h,i}(x)\,\de^2_{x_hx_l}u\,a_{k,j}(x)\cdot\nu_k
   \\[0.15cm]
   & \quad (\text{since $\de_{x_h}u\equiv 0$ on $\de\Omega$ for every $h = 1,\ldots,n$}) \\
   & \quad
   = \sum_{h,k,l = 1}^n
   \sigma_l\,x_l\,a_{h,i}(x)\,\de^2_{x_hx_l}u\,a_{k,j}(x)\cdot\nu_k =: (\bigstar).
    \end{align*}
   On the other hand, since $u\equiv \nabla u\equiv 0$ on $\de\Omega$,
   it is very easy to see that
   $$\frac{\de^2 u}{\de {x_h}\de {x_l}}u = \frac{\de^2u}{\de \nu^2}\,\nu_h\nu_l\qquad
   \text{on $\de\Omega$};$$
   thus, by crucially exploiting this information, we get
   \begin{align*}
    (\bigstar) & =
    \Big(\sum_{l = 1}^n\sigma_l\,x_l\,\nu_l\Big)\cdot
    \sum_{h,k = 1}^na_{h,i}(x)\,a_{k,j}(x)\,\Big(
    \frac{\de^2u}{\de \nu^2}\,\nu_h\nu_k\Big) \\
    & = \langle \Tv (x),\nu\rangle\cdot
    \sum_{h,k = 1}^na_{h,i}(x)\,a_{k,j}(x)\,\de_{x_h x_k}u
    \\
    & (\text{since $\de_{x_h}u\equiv 0$ on $\de\Omega$ for all $h = 1,\ldots,n$}) \\
    & = \langle \Tv (x),\nu\rangle\cdot
    \sum_{k = 1}^na_{k,j}(x)\,\de_{x_k}\Big(\sum_{h = 1}^na_{h,i}(x)\,\de_{x_h}u\Big)
    \\[0.15cm]
    & = \langle \Tv (x),\nu\rangle\cdot X_j(X_iu),
   \end{align*}
   which is precisely the claimed \eqref{eq.toproveBoundaryOrderdue}.
   \vspace*{0.07cm}
   
   With \eqref{eq.toproveBoundaryOrderdue} at hand, we can proceed with the proof
   of the theorem. Indeed, by combining the cited \eqref{eq.toproveBoundaryOrderdue} 
   with \eqref{eq.tosubstituteboundary}, we get
   \begin{align*}
    & \int_{\Omega}\big(q\mathcal{F}-\langle \nabla_p \mathcal{F}, \nabla_X u\rangle\big)\,\d x
   +
   \int_{\Omega}\mathrm{T}\big(x\mapsto\mathcal{F}(x,z,p,r)\big)(x,u,\nabla_Xu,\HH_X u)\, \d x 
   \\[0.1cm]
   & \qquad\quad - 2\sum_{i,j = 1}^m\int_{\Omega}\mathcal{F}_{r_{ij}}\,X_j(X_i u)\, \d x
   \\[0.1cm]
   & \quad = 
   \int_{\de\Omega}\langle \mathrm{T}(x),\nu\rangle\big(
   \mathcal{F}-\textstyle\sum_{i,j = 1}^mX_j(X_iu)\,\FF_{r_{ij}}\big)\,\d H^{n-1};
   \end{align*}
   from this, using assumptions (i)-(ii) and reminding that
   $\Omega$ is $\delta_\lambda$-star shaped
   with respect to the origin (see Definition \ref{def.starshaped}), we obtain
   \begin{align*}
   & q\FF - \langle \nabla_X,\FF_p\rangle  
   +\mathrm{T}\big(x\mapsto\mathcal{F}(x,z,p,r)\big)(x,u,\nabla_X,\HH_Xu)
   \\[0.1cm]
   &\qquad - 2
     \textstyle\sum_{i,j = 1}^mX_j(X_iu)\FF_{r_{ij}} = 0
     \qquad\text{for all $x\in\Omega$}.
     \end{align*}
   Now, according to assumption (iii), only three cases can occur:
   \medskip
   
    (a)\,\,$u\equiv 0$ on $\Omega$. In this case, \eqref{eq.toproveOrderdue} is
    satisfied and the proof is complete. \medskip
    
    (b)\,\,$\nabla_X u \equiv 0$ on $\Omega$. In this case,
    we have already recognized in the demonstration of Th\-eorem \ref{thm.mainNonexistOrder1}
    that $u$ must be constant in $\Omega$. Since, by assumption,
    $u\equiv 0$ on $\de\Omega$, we conclude that
    $u\equiv 0$ throughout $\Omega$, and 
    \eqref{eq.toproveOrderdue} is again satisfied. \medskip
    
    (c)\,\,$\HH_Xu \equiv 0$ on $\Omega$. In this case, bearing in mind
    the definition of $\HH_X$,
    we get
    $$\text{$\nabla_X(X_iu) \equiv 0$ on $\Omega$ for every $i = 1,\ldots,n$};$$
    from this, by arguing exactly as in the proof of Theorem \ref{thm.mainNonexistOrder1},
    we deduce that $\nabla_Xu$ must be constant on $\Omega$.
    On the other hand, since (by assumption) $\nabla u \equiv 0$ on $\de\Omega$,
    we necessarily have that $\nabla_X u \equiv 0$ on $\de\Omega$ as well; as a consequence,
    $$\nabla_Xu\equiv 0\qquad\text{on $\de\Omega$}.$$
    From (b) we then conclude that $u\equiv 0$ on $\Omega$,
    and \eqref{eq.toproveOrderdue} is again satisfied.
   \end{proof}
   \begin{remark} \label{rem.functionfproof}
    With reference to discussion before the statement of Theorem \ref{thm.nonexistenceOrderdue},
    we point out that identity 
    \eqref{eq.toproveBoundaryOrderdue} leads to the
    needed \eqref{eq.conditionboundaryB}: in fact, we have
    \begin{align*}
    & \textstyle\sum_{i,j = 1}^m\big(X_i(\mathcal{F}_{r_{ji}})\,\mathrm{T}u\
   -
   \mathcal{F}_{r_{ij}}\,X_j(\mathrm{T}u)\big)
   \langle X_j(x), \nu\rangle \\[0.1cm]
  & \qquad
   = -\textstyle\sum_{i,j = 1}^m
   \mathcal{F}_{r_{ij}}\,X_j(\mathrm{T}u)
   \langle X_j(x), \nu\rangle
   = -\langle \Tv (x),\nu\rangle\cdot\textstyle\sum_{i,j = 1}^m
   \FF_{r_{ij}}\,X_j(X_iu),
    \end{align*}
    and this is precisely identity \eqref{eq.conditionboundaryB} with the choice
    $$f(x,z,p,r) := -\sum_{i,j = 1}^mr_{ij}\FF_{r_{ij}}.$$
   \end{remark}

\end{document}